\documentclass{article}
\usepackage{amsmath}
\usepackage{amssymb}
\usepackage{amsthm,epsfig}
\usepackage{color}
\usepackage[T2A,T1]{fontenc}
\usepackage{lipsum}

\usepackage[utf8]{inputenc}
\usepackage[russian,english]{babel}

\newtheorem{theorem}{Theorem}[section]
\newtheorem{proposition}[theorem]{Proposition}
\newtheorem{lemma}[theorem]{Lemma}
\newtheorem{corollary}[theorem]{Corollary}

\newtheorem{definition}[theorem]{Definition}

\newtheorem{conjecture}[theorem]{Conjecture}
\newcommand\blfootnote[1]{%
  \begingroup
  \renewcommand\thefootnote{}\footnote{#1}%
  \addtocounter{footnote}{-1}%
  \endgroup
}

\begin{document}

\title{On the rigidity of rank  gradient in a  group  of  intermediate  growth}
\author{Rostislav Grigorchuk\footnote{The author was partially supported by Simons Collaboration Grant \# 527814}\\{Texas A\&M University}, \\Rostyslav  Kravchenko\blfootnote{To appear in Ukrainian Mathematical Journal v.70, no. 2 (2018)}\\{Northwestern University}}

\maketitle
 \begin{center}
{\it Dedicated to  A.M. Samoilenko on the occasion  of  his  80th  birthday}
\end{center}

\abstract{We  introduce and   investigate the rigidity  property of  rank  gradient  in  the  case   of  the  group $\mathcal  G$ of  intermediate  growth constructed by  the  first  author  in  \cite{grigorch:burnside}.  We  show  that  $\mathcal  G$   is  normally $(f,g)$-RG  rigid   where  $f(n)=\log(n)$  and  $g(n) =\log(\log(n)).$ }

\section {Introduction} A  group  is said to be  residually  finite   if  it  has  sufficiently  many  subgroups of  finite  index  so  that the  intersection of  them  is trivial. This  is  important  class  of  groups studied throughout more  than eight  decades by  various  tools  and  means. Residually  finite  groups  are the  reach  source of  examples  in group  theory.  In particular  they  are  often used in the  three main branches  of   modern  group  theory:   geometric  group  theory, asymptotic group theory  and  measured group  theory.  Such  groups   have   realization  by   actions  on  spherically  homogeneous  rooted  trees as  indicated  in  \cite{grigorch:jibranch,grigorch_nek_sus00},  which, in many  cases,  gives a possibility  to  study them  and  their  subgroup  structure  using   the  structure  of  the  tree.  They  are  also  closely  connected  to  the  theory  of  profinite  groups.

A  very  important  invariant  of  residually  finite  group  is a  subgroup  growth  introduced  by  F.~Grunewald, D.~Segal  and  G.~Smith  \cite{grunewald_seg_smith:88}  and  studied by many  researches  (see a  comprehensive  book  \cite{lubotzky_segalbook} and  the  literature  therein  on this  subject).   Recently another  asymptotic  characteristics  of  residually finite  groups  were  introduced with  the  focus  on the notion of  rank gradient.

 The rank  gradient  of  a finitely generated residually  finite  group  $G$  is defined  as

\begin{equation}
\label{rankgrad}
 RG(G)=\inf \frac{d(H)-1}{[G:H]}
 \end{equation}

where  the  infimum  is  taken  over  all   subgroups $H$ in $G$  of  finite  index and  $d(H)$ is the  rank  of  $H$ (i.e. the  minimal  number  of  generators  of $H$). It  is a  finite  number because  a  subgroup of finite  index  in a  finitely  generated  group  is  finitely  generated, and the first  question  that  arises  is  whether $RG(G)=0$ or not.

This  notion, as  well   as   the  notion  of  the  rank  gradient  relative  to  the  descending  chain of  subgroups (defined  by (\ref{rankgrad2}) ), were  introduced  for  the  first  time  by  M.~Lackenby \cite{lackenby:rank05}   with motivation  from  3-dimensional  topology.   Since   that  the  rank  gradient  and  its  variations  were  intensively  studied  \cite{abert_nik:rank,grigorch_krav_lamp14,grigorch:MIAN11,grigorch_nek_sus00,grigorch:solved,lackenby:rank05}.

The  definition  (\ref{rankgrad})
  can be  modified  in various directions.  Instead  of  $\inf$ one  can  consider  $\sup$,  instead of  all  subgroups one can consider  only  normal  subgroups,  or  subgroups    with   index  a  power  of a  prime  number $p$,  etc.     Another  direction   of modifications  is  to  consider descending sequence  $H_n, n=1,2,\dots$ of  subgroups  of  finite  index  and  associated  sequence $rg(n)$  of  numbers defined as

\begin{equation}
\label{rankgrad2}
 rg(n)=RG(G,\{H_n\})= \frac{d(H_n)-1}{[G:H_n]}
 \end{equation}

and  its  limit

\begin{equation*}
\lim_{n \to \infty} rg(n),
 \end{equation*}

if  it exists,  or the upper  and the lower  limits

\begin{equation*}
 \limsup_{n \to \infty} rg(n),
 \end{equation*}

\begin{equation*}
 \liminf_{n \to \infty} rg(n)
 \end{equation*}
otherwise.

It  is  known  \cite{lackenby:rank05,abert_nik:rank}  that if  $G$  is  amenable (the  notion  of  amenable  group  was  introduced  by  von Neumann \cite{vonNeumann:1929} and  by  Bogolyubov  in topological  case  \cite{bogolyubov:amenable39}), and  $\{H_n\}$   is a  sequence  of    subgroups  of  finite  index  satisfying  some technical  condition (sometimes  called the Farber  condition, it  is  equivalent  to  the  essential  freeness  of the  action of  the group on  the boundary of  coset  tree \cite{grigorch:MIAN11}), then  the  limit (\ref{rankgrad2}) exists and  is  equal  to  $0$.  The  example  of  the  lamplighter  group  $\mathcal L=\mathbb{Z}/2\mathbb{Z}\wr \mathbb{Z}$  show  that $rg(n)$  may  have  arbitrary  fast  decay,  just  because  the  group   has a  subgroup  of  index  2  isomorphic  to  itself,  so iterating  this  fact  one  gets  a  descending  sequence of groups of  growing   index power  of  $2$  but a fixed  rank $=2$ (see for instance \cite{grigorch_krav_lamp14}).

We  suggest  the  following  definition.  Let  $f(n),g(n)$  be two  increasing  functions of natural  argument  $n$  taking    values in $\mathbb N$ and  having the  limit $ \infty$ when $n \to \infty$.
\begin{definition}(a) A  finitely  generated  group   $G$  is   $(f,g)-RG$-rigid   if   there  is $C\in \mathbb N$ such  that for  every  subgroup  $H<G$ of  finite  index

\begin{equation*}
g(d(H))<Cf(C[G:H])    \text{ and }  f([G:H])<Cg(Cd(H)).
\end{equation*}

(b)   $G$ is  normally $(f,g)-RG$-rigid  if previous  inequalities   hold for  each  normal  subgroup $H\triangleleft G$ of  finite  index.
\end{definition}

For  instance the  free  group  $F_r$  of rank  $r\geq 1$  is  $(f,g)-RG$-rigid  where  $f(n)=g(n)=n$    as  in  this  case  the  ratio  $(d(H)-1)/([F_r:H])$ is  constant  and  equal  to  $r-1$.  Also  $(n,n)-RG$-rigid are  all  groups  with $RG(G)>0$.  For  finitely  presented  groups  this  hold  if and  only  if  $G$  is ``large''  in  the  sense  of  S.~Pride,  i.e.   contains a subgroup  of  finite  index  that  surjects  onto a  noncommutative  free  group,  as  shown  in  \cite{lackenby:rank05}.

We  present  the  following two  results.  Let  $\mathcal{G}=\langle a,b,c,d \rangle$  be  an  infinite 2-group  constructed  by  the  first  author in \cite{grigorch:burnside}.
Recall  that  it  has  intermediate  growth  between polynomial  and  exponential  and has  many  other  interesting  properties \cite{grigorch:degrees,grigorch:jibranch,grigorch:solved}.   $\mathcal{G}$  has  many  other  ways  to  be  defined  but  for  us  it  will  be  important  that  it  has a natural  action  by  automorphisms  of  a  rooted binary  tree $\mathcal T$  shown by  Figure \ref{binary}, as   explained,  for  instance,  in  \cite{grigorch:jibranch}.

\begin{figure}
\begin{center}
\epsfig{file=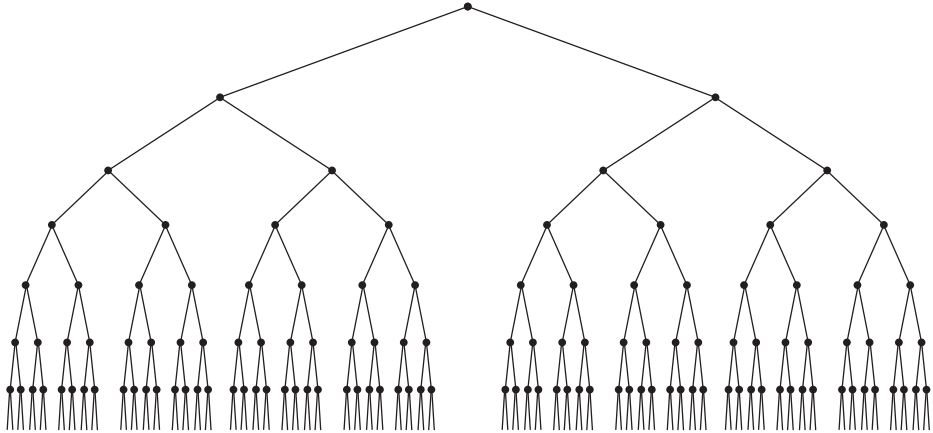 , width=350pt}
\caption{Binary Tree\label{binary}}
\end{center}
\end{figure}

\begin{theorem}\label{th1} The   group $\mathcal{G}$  is normally  $(f,g)-RG$-rigid  with

$$f(n)=\log\log(n), g(n)=\log(n),$$.

Moreover,  there there  is a  constant  $D>1$  such  that
\begin{equation}
\label{main}
\frac{1}{D}\log(d(H))\leq \log\log([\mathcal{G}:H])\leq D\log(d(H))
\end{equation}
hold  for  every nontrivial  normal  subgroup  $H\lhd \mathcal G$.
\end{theorem}

\begin{conjecture} The   group $\mathcal{G}$  is  $(f,g)-RG$-rigid  with the  same  functions $f,g$  as  in  previous  theorem.
\end{conjecture}

If  the  conjecture  is  true  then  we have an interesting  rigidity  property concerning  the  rank and  index  of finite index subgroups. At  least  we  have  this  property  for  normal  subgroups as shows the above theorem.  Recall  that all  normal subgroups in $\mathcal G$ have finite  index  because  $\mathcal  G$  is  just-infinite (i.e. infinite, but  every  proper   quotients  finite) as  shown  in  \cite{grigorch:jibranch}. Interestingly,  the group  $\mathcal G$  was  used  by  M.~Lackenby  in  \cite{lackenby:rank05} to demonstrate  some  phenomenon that  may  hold for  the  rank  gradient.  The  present  article  develops   the observation made by M.~Lackenby concerning  the  rank  gradient  in $\mathcal G$.

At  the  moment  we  are  able  only  to  confirm  the  conjecture  for   important  subclass  of  subgroups  of  finite  index in $\mathcal G$,  namely for stabilizers of  vertices of  the  binary  rooted  tree  $\mathcal T$ on  which  the  group $\mathcal G$ acts. The  vertices  of  $\mathcal T$  are  in  bijection  with finite  words  over  binary  alphabet $\{0,1\}$. Let  $v$   be a  vertex  and $st_{\mathcal{G}}(v)$  be  its stabilizer which  has index  $2^n$ in $\mathcal{G}$  if  $v$  is  a  vertex  of  level  $n$.

\begin{theorem}\label{th2}   There  is a  constant $D$  such  that the  inequalities  (\ref{main})  hold  for  all subgroups $H=st_{\mathcal{G}}(v)$
where $v$  run over the  set of  vertices    in $\mathcal T$. In fact

\[
\frac{d(H)-1}{[\mathcal{\mathcal{G}}:H]}=\frac{n+3}{2^n}
\]
if $v$  is a vertex  of  level $n\geq 2$.
\end{theorem}

Observe that this  result  is  announced  in \cite{grigorch_krav_lamp14}.



\section{The  group $\mathcal G$}

We  recall  some  basic  facts  about  the  group  $\mathcal G$  and  its  subgroups.   $\mathcal  G$  can be  defined  as a  group of  automorphisms  of a  rooted  binary  tree $\mathcal T$  shown  in Figure \ref{binary} (the  root,   corresponding  to  the  empty  word,   is a fixed  point  for  the action). The  generators   $a,b,c,d$ of $\mathcal G$  are  involutions, the  elements  $b,c,d$ commute  and  together  with the identity  element  they  constitute the  Klein  group  $\mathbb{Z}_2\times\mathbb{Z}_2$. The  stabilizer  of  the  first  level $H=st_{\mathcal G}(1)$ is a  subgroup  of  index  2  in  $\mathcal  G$  generated  by  elements  $b,c,d,b^a,c^a,d^a$ (where  $x^y=x^{-1}yx$), and  the   restrictions of  $H$  on  the  left  and  right  subtrees $\mathcal {T}_0,\mathcal {T}_1$ with  the  roots  at  vertices  $0,1$  determine surjective  homomorphisms $\varphi_0,\varphi_1:H\rightarrow \mathcal{G}$.  The direct  product    $\psi=\varphi_0\times \varphi_1$ of  them  determines the embedding $H\rightarrow  \mathcal G \times \mathcal G$  and  acts on generators as:

 \begin{equation}
 \label{recursion1}
 b\rightarrow (a,c), c \rightarrow (a,d), d\rightarrow (1,b), b^a\rightarrow (c,a), c^a\rightarrow (d,a), d^a\rightarrow (b,1)
 \end{equation}

   Together   with  the  information  that  the  generator  $a$    permute  the subtrees   $\mathcal {T}_0,\mathcal {T}_1$ (without  extra  action  inside them),  this  uniquely  determines the  group   $\mathcal  G$.

 The  action   of  $\mathcal  G$  on $\mathcal  T$ also  can  be  described  by  the  following  recursive  rules:
 \begin{equation*}
 \begin{array}{ll}
 a(0w)=1w,& a(1w)=0w,\\
 b(0w)=0a(w),& b(1w)=1c(w),\\
 c(0w)=0a(w),& c(1w)=1d(w),\\
 d(0w)=0w,& d(1w)=1b(w),
 \end{array}
 \end{equation*}

where   $w \in \{0,1\}^{\ast}$,   and  $\{0,1\}^{\ast}$  denotes the  set  of  all   finite words  over  the  binary  alphabet.  The  important  property of  the  action  of  $\mathcal  G$  on $\mathcal  T$  is  level  transitivity,  i.e. transitivity  of  the  action on  each  level  $V_n=\{0,1\}^n$.

Additionally  to  the  stabilizers  $st_{\mathcal G}(n)$ of  levels  $n=1,2,\dots$, an important  descending series  of  normal subgroups  is  the  series  of  rigid  stabilizers $rist_{\mathcal G}(n)$    which   are  subgroups    generated  by  rigid  stabilizers  $rist_{\mathcal G}(v), v\in \{0,1\}^n $  of  vertices  of the $n$th  level,   and     $rist_{\mathcal G}(v)$  is a  subgroup in  $\mathcal G$  fixing  vertex  $v$  and  consisting  of  elements acting trivially  outside the   subtree  $\mathcal  T_v$ in  $\mathcal  T$  with a  root at  $v$. The  rigid  stabilizers  of distinct  vertices of  the  same  level  commute  and  are  conjugate  (because  of the  level  transitivity).  Thus  algebraically  the $rist_{\mathcal G}(n)$   is a  direct  product  of  copies  of  the  same  group (which  may  depend  on  the  level  $n$  in general   case).  Observe  that   $\{st_{\mathcal G}(n)\}_{n=1}^{\infty}$ and  $\{rist_{\mathcal G}(n)\}_{n=1}^{\infty}$  are  descending  chains  of  normal  subgroups of  finite  index with trivial  intersection. The  structure  of  groups   $st_{\mathcal G}(n)$ and $rist_{\mathcal G}(n)$  is  well  understood  and   described in  \cite{bartholdi_g:parabolic}.

 Let  $B=\langle b \rangle^{\mathcal G}$  be  a normal  closure  of  generator  $b$  and  $K=\langle(ab)^2\rangle^{\mathcal G}$.  For   each $n$ there  is a  natural  embedding $\psi_n: st_{\mathcal G}(n)\rightarrow \mathcal G \times \dots \times \mathcal G$ into a direct  product  of  $2^n$  copies of $\mathcal G$  which  is the $n$th  iteration of  the  embedding $\psi$ and has a  geometric  meaning  of  the attaching  to  the  element  $g \in st_{\mathcal G}(n)$  the  $2^n$-tuple $(g_1,\dots,g_{2^n})$  of  its  restrictions  on the  subtrees  with  the  roots  at the $n$th  level. Instead  writing  $$\psi_n(g)=(g_1,\dots,g_{2^n})$$  we  will  write $$ g=(g_1,\dots,g_{2^n}).$$  In  particular  the  relations (\ref{recursion1})  can  be  rewritten   as  $b=(a,c), c=(a,d),d=(1,b),b^a=(c,a),c^a=(d,a),d^a=(b,1)$.

 The  important  facts  about  groups    $st_{\mathcal G}(n)$ and $rist_{\mathcal G}(n)$   are  that  the  $\psi_{n-3}$  image of $st_{\mathcal G}(n)$  has  the  decomposition
 \begin{equation}
 \label{strange}
 st_{\mathcal G}(3)\times st_{\mathcal G}(3) \times \dots \times st_{\mathcal G}(3)
 \end{equation}
 (products of  $2^{n-3}$ copies of $st_{\mathcal G}(3)$)   when  $n\geq  4$,   and the $\psi_{n}$ image of $rist_{\mathcal G}(n)$ has  the  decomposition

 \begin{equation}
\label{strange1}
  K\times K \times \dots \times K
 \end{equation}

(product  of  $2^n$  copies  of  $K$)  when  $n\geq 2$ (see \cite{bartholdi_g:parabolic}). We  will  use  later  the notations 
 $K_n$  for $rist_{\mathcal G}(n)$ when  $n\geq 2$
and  keep  in mind  the  decompositions (\ref{strange})  and (\ref{strange1}).   We  also  denote  by  $K_1$  a  subgroup in $st_{\mathcal G}(1)$  whose $\psi$-image  is  $K\times K$.  The  group    $\mathcal  G$  is  regularly  branched  over  $K$  as   $K_1$  is a  subgroup  of  $K$,  and  $[\mathcal {G}:K_1] <  \infty$.   Thus     $\mathcal G$  contains   subgroups  shown  by  Figure  \ref{regular}.  To  each  level  $n$  corresponds  a  group  $K_n$   that  fixes  each  vertex  $v$ of  this  level  and  whose restriction  to  the   subtree $\mathcal{T}_v$ with  the  root   $v$  is  $K$ (if  to  identify   $\mathcal{T}_v$  with   $\mathcal{T}$).  Moreover $K_n$  is a  direct  product  of  these  projections.

\section{Proof  of  Theorem  \ref{th1}}

We     begin this  section  with  reminding  that  the  group $\mathcal G$  is branch  group  as  defined   in  \cite{grigorch:jibranch},   because  it  acts level transitive  on  the  tree  $\mathcal T$  and   rigid  stabilizers $rist_{\mathcal G}(n), n=1,2,\dots$  have  finite  index  in $\mathcal G$.
The  branch  structure  of   $\mathcal G$  that  we  will  use  is  given  by  the  Figure \ref{regular}  and  was  mentioned  in  previous  section.
The  proof of  the theorem \ref{th1}  is  based  on the  following

\begin{proposition}\label{normal0} Let $N\lhd \mathcal G$  be a  nontrivial  normal  subgroup.  Let $n$  be  a  smallest nonnegative  integer  such  that   $N < st_{\mathcal G}(n)$  but  $N$ is not a subgroup  of  $st_{\mathcal G}(n+1)$. Then

(a)
\begin{equation*}
st_{\mathcal G}(n+6) < N .
\end{equation*}
(b)  If  $n\geq  4$  then

\begin{equation*}
rist_{\mathcal G}(n+3) < N < rist_{\mathcal G}(n-3)
\end{equation*}

\end{proposition}
\begin{proof} For  the  part  (a) we first   address  the  reader to the  proof of the theorem 4  given in \cite{grigorch:jibranch}  as  we  will  follow  the  same  line in  our  arguments.  First,  let us   make  a  comparison  of  notations  used  in \cite{grigorch:jibranch} and  here.  The  branch  structure  for  the  group $\mathcal G$   is  given    in  our  case  by  the  pair  $(\{L_n\},\{H_n\})$   where  $L_n=K$  and $H_n=K_n$ for  all  $n$.  The  reader  should  keep  in  mind  the  picture given  by the figure \ref{regular}.

\begin{figure}
\begin{center}
\epsfig{file=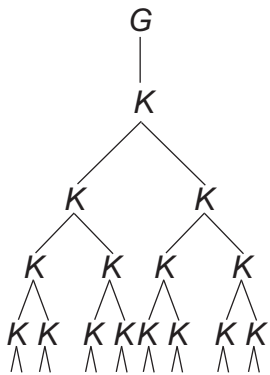 , width=100pt}
\caption{Structure of branching subgroups\label{regular}}
\end{center}
\end{figure}

Also,  the normal  subgroup in the  statement  of theorem  4  is  denoted  by  $P$  while  in the proposition  under  consideration  it  is  denoted  $N$.   The  proof of   theorem  4 from \cite{grigorch:jibranch}  (modulo  of the  change  of  notations)  shows  that if   $N< st_{\mathcal G}(n)$  but $N$ is not a subgroup of $st_{\mathcal G}(n+1)$, then  $N$ contains commutator subgroup  $K'_{n+1}=[K_{n+1},K_{n+1}]$ (recall  that $rist_{\mathcal G}(n)=K_n$  if  $n>1$).  By  proposition  9  from \cite{grigorch:jibranch}
$K'=K_2$,  hence   $K'_{n+1}=K\times K \times \dots \times K$ ($2^{n+3}$  factors)  so  $K'_{n+1}=K_{n+3}$,  and  $N > K_{n+3}$, which gives the lower inclusion in part (b).

Also  by  proposition  9  from \cite{grigorch:jibranch}   we  have  the  inclusion  $K>st_{\mathcal G}(3)$   from which,  together  with  the fact  presented  by  factorization (\ref{strange})   we  conclude  that $K_{n+3}>st_{\mathcal G}(n+6)$ and thus   $N>st_{\mathcal G}(n+6)$.  The  part  (a)  is thus established.
To  get  the upper inclusion in part  (b)  under  assumption  that  $n>3$ we observe  that  the  inclusion $K>st_{\mathcal G}(3)$  and  factorization (\ref{strange})  imply  that   $K_{n-3}> st_{\mathcal G}(n)$, therefore  we  are  done.
\end{proof}

The  group  $\mathcal  G$ is not  virtually  cyclic  (by many  reasons,  for  instance  because  it  is  finitely  generated  infinite  torsion  group).  Let  $N\lhd G$  be  a  nontrivial  normal  subgroup.  It  is  automatically  of  finite  index,  as  $\mathcal  G$  is just-infinite   group
\cite{grigorch:jibranch}  (i.e.  infinite  group  with  every  proper quotient  finite).  In  fact  we  can  assume from  the  beginning  that  $N$  is a normal subgroup of  finite  index  when  proving  theorem \ref{th1} (so  the  just-infiniteness property is  not  needed).  If  $n$ in the  statement  of  proposition  \ref{normal0} is less  than  $4$,   then  there  are  only  finitely  many  subgroups in $\mathcal  G$   containing $st_{\mathcal G}(5)$, their ranks are $\geq 2$,  and  so there  is a  constant $D$  satisfying  the  condition  of  the  theorem \ref{th1}.
Therefore  we  can  assume  that  $n\geq 4$.  Now apply  part  (b)  of  the proposition.

The  quotient {\color{black}$rist_{\mathcal G}(n-3)/ rist_{\mathcal G}(n+3)=K_{n-3}/K_{n+3}$ } is  isomorphic to

$$A:=(K/K_6)^{\color{black}2^{n-3}}=K/K_6 \times \dots \times K/K_6,$$
($2^{n-3}$  factors). The  group   $K/K_6$ is  a  finite  $2$-group  of  certain  nilpotency class $l$ (for us it  is not  important  the  exact  value of $l$).  Therefore $A$  is nilpotent of  the  class  $l$  as  well and  it  is  generated  by  not more  than  $3\cdot2^{n-3}$ elements  as  $K$ is  3-generated  group \cite{grigorch:jibranch}.  It  is  well  known  that  a   subgroup  of  finitely  generated nilpotent group  is  finitely  generated  and  there is a  universal  upper  bound  on  the  ranks  of  subgroups  of  nilpotent  group  $G$ in  terms  of  $d(G)$  and  class of  nilpotency $c$  of  $G$.

We  use  the  most  simple  upper  bound  given by  the  following  lemma
\begin{lemma}\label{upper1}  Let  $G$ be  a  finitely  generated  nilpotent  group of  class $c$.  Then  for  every  subgroup  $H<G$ the  upper  bound
\begin{equation*}
d(H)< d(G)^{c}
\end{equation*}
holds.
\end{lemma}
\begin{proof}
If $c=1$ then $G$ is abelian and thus $d(H)\leq d(G).$ For $G$ of class $c$ let $\gamma_1(G)=G$ and $\gamma_{i+1}(G)=[G,\gamma_i(G)]$ $i=1,2,\dots$ be the elements of the lower central series. Suppose $G$ is generated by set $S.$ Each factor $\gamma_i(G)/\gamma_{i+1}(G)$ is generated by the iterated commutators $[s_1,[s_2,\dots ]]$ of length $i+1$, where $s_j\in S.$  Thus $d(\gamma_i(G)/\gamma_{i+1}(G))\leq d(G)^{i+1}$. 
Denote $H_i=H\cap \gamma_i(G).$ Then $H_i\gamma_{i+1}(G)/\gamma_{i+1}(G)$ is an abelian group, and thus
$d(H_i\gamma_{i+1}(G)/\gamma_{i+1}(G))\leq d(\gamma_i(G)/\gamma_{i+1}(G))\leq d(G)^{i+1}.$ Note that 
\[
d(H)\leq d(H\gamma_1(G)/\gamma_1(G))+d(H\cap \gamma_1(G))\leq d(G)+d(H_1).
\]
Applying this iteratively we obtain
\begin{equation*}
\begin{aligned}
& d(H)\leq d(G)+d(H_1)\leq d(G)+d(H_1\gamma_2(G)/\gamma_2(G))+d(H_2)\leq \\
& d(G)+ d(\gamma_1(G)/\gamma_2(G))+d(H_2)\leq \\
& d(G)+d(G)^2+d(H_2)\leq\cdots\leq d(G)+\dots+d(G)^{c-1}\leq d(G)^{c}
\end{aligned}
\end{equation*}
\end{proof}

Using  this  lemma  we  get

$$d(N)< (3\cdot2^{n-3})^{l+2}=3^{l+2}2^{(n-3)(l+2)}<a_12^{a_2n}$$
for   some  positive  constants  $a_1,a_2$.

Now  we  are  going  to  give a  lower  bound for  $d(N)$.  We  factorize  the inclusions $K_{n+3}<N<K_{n-3}$  by $K'_{n+3}$   getting

$$K_{n+3}/K'_{n+3}<N/K'_{n+3}<K_{n-3}/K'_{n+3}$$

The  group $K_{n-3}/K'_{n+3}$  is a  direct  product of $2^{n-3}$  copies  of  the  group $K/K_8$  as $K=K_2$.    $K/K_8$ is a  finite  2-group.  Let $s$ be its  nilpotency  class,  so $K_{n-3}/K'_{n+3}$ also  has  nilpotency  class $s$ .  As $\bar{N}:=N/K'_{n+3}$ is a  subgroup of  $K_{n-3}/K'_{n+3}$ its  class  of  nilpotency  is  $\leq  s$.

The  group  $K_{n+3}/K'_{n+3}$  is a  direct  product  of $2^{n+3}$  copies  of  the  group $K/K'=K/K_2$.  Let  $t=d(K/K')$  (in fact  $t=3$).
Then $d(  K_{n+3}/K'_{n+3})=t2^{n+3}$.   Using  lemma \ref{upper1} we  conclude

 $$d(  K_{n+3}/K'_{n+3})=t2^{n+3}\leq (d(N))^{s+2}$$
from  which  we  conclude  that  there  are positive  constants $a_3,a_4$  such that

 $$ a_32^{a_4n} \leq  d(N)$$
Now  using  the  part  (a) of  the  Proposition \ref{normal0} we  provide upper  and  lower  bounds  for  the  index  $[\mathcal{G}:N]$.
 For  this  purpose  we  use  the  fact  that

$$[\mathcal{G}:st_{\mathcal G}(n)]= 2^{5\cdot2^{n-3}+2}$$
as  shown  at  the  end  of  the  proof  of  theorem  14  from  \cite{grigorch:jibranch}. Hence  part  (a)  of  the  proposition lead  us  to  the  existence  of  positive  constants  $a_5,a_6$  and  constants $a_7,a_8$  such  that
\begin{equation*}
2^{a_52^n+a_7}\leq  [\mathcal {G}:N]\leq 2^{a_62^n+a_8}
\end{equation*}
Taking  the  double  logarithm with  base  $2$ of  this  inequalities  and  applying the  same  logarithm  to  the previously  obtained  inequalities

$$ a_32^{a_4n} \leq  d(N) \leq  a_12^{a_2n}$$
a  simple  calculus  finishes  the  proof  of  the  theorem.

\section{Proof  of  Theorem  \ref{th2}}

We  begin  the  proof  of  Theorem \ref{th2}. The groups $R_n,Q_n$ and $P_n$ which are defined below are the finite index analogs of groups $R,Q,P$ introduced and studied in \cite{bartholdi_g:parabolic}. The recursive relations between them are analogous to the corresponding relations between $R,Q,P$ given in \cite{bartholdi_g:parabolic} by Theorems 4.4 and 4.5. The proof is  based  on  a number  of  computations   that  we  split  in  propositions  and lemmas. 

Let us  introduce  the  elements $t=(ab)^2$, $u=(bada)^2=(t,1)$, $v=(abad)^2=(1,t)$.   The  direct  computation  gives the  result of  their  conjugation by  generators, together with the conjugation of  $(ac)^4,$ as  shown in the next  subsections:

\subsection{Conjugates of $t,u,v$}

\begin{equation}\label{t}
\begin{aligned}
& t^a=t^{-1} \\
& t^b=t^{-1} \\
& t^c=t^{-1}v \\
& t^d=v^{-1}t \\
& t^{dd^a}=v^{-1}tu \\
\end{aligned}
\end{equation}

\begin{equation}\label{u}
\begin{aligned}
& u^a=v \\
& u^b=u^{-1} \\
& u^c=u^{-1} \\
& u^d=u \\
& u^{dd^a}=u^{-1} \\
\end{aligned}
\end{equation}

\begin{equation}\label{v}
\begin{aligned}
& v^a=u \\
& v^b=t^{-1}v^{-1}t \\
& v^c=t^{-1}vt \\
& v^d=v^{-1} \\
& v^{dd^a}=v^{-1} \\
\end{aligned}
\end{equation}

\subsection{Conjugates of $(ac)^4$}

Denote $x_0=(ac)^4$

\begin{equation*}
\begin{aligned}
& x_0^a=x_0 \\
& x_0^b=(1,u)x_0 \\
& x_0^c=x_0 \\
& x_0^d=(1,u)x_0 \\
& x_0^{dd^a}=(u,u)x_0. \\
\end{aligned}
\end{equation*}
Now let us  introduce  more  elements  and show  the  result  of  their  conjugation.

\subsection{Definition and conjugates of $x_m,u_m,v_m$}
Recall that $rist_{\mathcal G}(1^m)$ is the subgroup of $\mathcal G$ that fixes the vertex $1^m$ together with all vertices that do not start with $1^m.$  It is easy to check that $(ac)^4,u$ and $v$ belong to $K,$ and that $\psi(K)\supset K\times K.$ It follows that there are such $x_m,u_m$ and $v_m$ in $rist_{\mathcal G}(1^m)$ that $x_m(1^mw)=1^m(ac)^4(w),$ $u_m(1^mw)=1^mu(w)$ and $v_m(1^mw)=1^mv(w).$  Note that $x_0=(ac)^4$. 
We compute their conjugates below (the proof is by induction, since $x_m=(1,x_{m-1})$, $u_m=(1,u_{m-1})$ and $v_m=(1,v_{m-1})$):

\begin{equation}\label{xxm}
x_m^b=\left\{
\begin{array}{ll}
u_{m+1}x_m &\text{ if }3|m \\
x_m &\text{ if }3|(m-1) \\
u_{m+1}x_m &\text{ if }3|(m-2)
\end{array}\right.
\end{equation}

\begin{equation*}
x_m^c=\left\{
\begin{array}{ll}
x_m &\text{ if }3|m \\
u_{m+1}x_m &\text{ if }3|(m-1) \\
u_{m+1}x_m &\text{ if }3|(m-2)
\end{array}\right.
\end{equation*}

\begin{equation}\label{xxxm}
x_m^d=\left\{
\begin{array}{ll}
u_{m+1}x_m &\text{ if }3|m \\
u_{m+1}x_m &\text{ if }3|(m-1) \\
x_m &\text{ if }3|(m-2)
\end{array}\right.
\end{equation}
We have that 
\begin{equation}\label{x0}
x_0^{dd^a}=(u,u)x_0,
\end{equation}
and for $m>0$ 
\begin{equation}\label{xm}
x_m^{dd^a}=(1,x_{m-1}^b)=\left\{
\begin{array}{ll}
u_{m+1}x_m &\text{ if }3|m \\
u_{m+1}x_m &\text{ if }3|(m-1) \\
x_m &\text{ if }3|(m-2)
\end{array}\right.
\end{equation}

\begin{equation}\label{x1x0}
x_1^{x_0}=(1,x_0^{dd^a})=(1,1,u,u)(1,x_0)=(1,1,u,u)x_1
\end{equation}
When $n+1<m$ we have,

\begin{equation}\label{more}
x_m^{x_n}=\left\{
\begin{array}{ll}
x_m &\text{ if }3|(m-n) \\
u_{m+1}x_m &\text{ if }3\not{|}(m-n) \\
\end{array}\right.
\end{equation}
Here is the list of conjugates of $u_m$ and $v_m$:

\begin{equation}\label{uum}
u_m^b=\left\{
\begin{array}{ll}
u_m^{-1} &\text{ if }3|m \\
u_m^{-1} &\text{ if }3|(m-1) \\
u_{m} &\text{ if }3|(m-2)
\end{array}\right.
\end{equation}

\begin{equation}\label{um}
u_m^{dd^a}=\left\{
\begin{array}{ll}
u_m &\text{ if }3|m \\
u_m^{-1} &\text{ if }3|(m-1) \\
u_{m}^{-1} &\text{ if }3|(m-2)
\end{array}\right.
\end{equation}

\begin{equation*}
v_m^b=\left\{
\begin{array}{ll}
v_{m-1}^{-1}v_m^{-1}v_{m-1} &\text{ if }3|m \\
v_{m-1}^{-1}v_mv_{m-1} &\text{ if }3|(m-1) \\
v_{m}^{-1} &\text{ if }3|(m-2)
\end{array}\right.
\end{equation*}

\begin{equation}\label{vm}
v_m^{dd^a}=\left\{
\begin{array}{ll}
v_m^{-1} &\text{ if }3|m \\
v_{m-1}^{-1}v_m^{-1}v_{m-1} &\text{ if }3|(m-1) \\
v_{m-1}^{-1}v_m v_{m-1} &\text{ if }3|(m-2)
\end{array}\right.
\end{equation}
In the  next  two  subsections  we  introduce   sequences  of  subgroups  $R_n$  and $Q_n, n=1,2,\dots$ and  prove  some  structural results  about  them.
\subsection{Groups $R_n$}

Let $R_1=K=\langle t, u, v\rangle$. Let $R_{n}=(K\times R_{n-1})\{1,(ac)^4\}$ for $n\geq 2$. Then

\begin{proposition}\label{rn}
\[
R_2=\langle x_0, u_0, u_1, v_0, (u,u)\rangle,
\]
\[
R_n=\langle x_0,\dots, x_{n-2},u_0,u_1,u_2,v_{n-2},(u,u)\rangle
\]
for $n\geq 3$.
\end{proposition}
\begin{proof}
Since $R_2=(K\times K)\{1,x_0\}$, it follows that
\[
R_2=\langle x_0, (t,1), (u,1), (v,1), (1,t), (1,u), (1,v)\rangle.
\]
 Now, by \eqref{t} $t^{dd^a}=v^{-1}tu$, and since $x_0=(dd^a,dd^a)$, we have that $(t,1)^{x_0}=(v,1)^{-1}(t,1)(u,1)$, and the same for $(1,t)$. So we can discard $(v,1),(1,v)$ from the list. Notice also that multiplying $(u,1)$ by $(1,u)$ we can replace $(u,1)$ with $(u,u)$. Since $u_0=u=(t,1)$, $u_1=(1,u)$, and $v_0=v=(1,t)$, we obtain the generators for $R_2$.

 We have $R_3=(K\times R_2)\{1,x_0\}$, hence $R_3$ is generated by
 \[
 R_3=\langle x_0, (t,1), (u,1), (v,1), x_1, u_1, u_2, v_1,(1,1,u,u)\rangle.
 \]
 Where $x_1=(1,x_0), u_1=(1,u_0), u_2=(1,u_1)$ and $v_1=(1,v_0).$
 
 In the same way as for $R_2$ we can discard $(v,1)$ and replace $(u,1)$ by $(u,u)$. By \eqref{x1x0} $x_1^{x_0}=(1,1,u,u)x_1$, and so we can discard $(1,1,u,u)$, and we are done.

 Now suppose
 \[
R_n=\langle x_0,\dots, x_{n-2},u_0,u_1,u_2,v_{n-2},(u,u)\rangle
\]
Then from the formula $R_{n+1}=(K\times R_{n})\{1,(ac)^4\}$ we obtain that $R_{n+1}$ is generated by
\[
\langle x_0,(t,1),(u,1),(v,1),x_1,\dots,x_{n-1},u_1,u_2,u_3,v_{n-1},(1,1,u,u)\rangle.
\]
As above, we can discard $(v,1)$ and $(1,1,u,u)$ and replace $(u,1)$ with $(u,u)$.
It is left to note that $x_2^{x_0}=u_3x_2$ by \eqref{more}, so $u_3=[x_0,x_2]$ and hence we can discard it from the list.
\end{proof}

For any group $G$ define $G^{(2)}=G/G(X^2)$, where $G(X^2)$ is the subgroup generated by all squares of elements in $G$. Then $G^{(2)}$ is elementary abelian $2$-group, and so $\dim_{\mathbb{F}_2}G^{(2)}$ is defined.

\begin{proposition}
$\dim_{\mathbb{F}_2}R_n^{(2)}=n+4$ for $n\geq 3$, $\dim_{\mathbb{F}_2}R_2^{(2)}=5$, and $\dim_{\mathbb{F}_2}K^{(2)}=3$.
\end{proposition}
\begin{proof}
The equality for $K$ follows from the theory of the group $\mathcal{G},$ see \cite{bartholdi_g:parabolic,grigorch:jibranch,grigorch:solved}.

To prove the rest, we need the following obvious lemma
\begin{lemma}\label{lem}
Suppose $G=H\rtimes(\mathbb{Z}/2\mathbb{Z})$. Let $\alpha:H^{(2)}\rightarrow H^{(2)}$ be the operator on $H^{(2)}$ induced by the action of $\mathbb{Z}/2\mathbb{Z}$. Then $G^{(2)}=H^{(2)}/(1+\alpha)(H^{(2)})\oplus\mathbb{Z}/2\mathbb{Z}$.
\end{lemma}

Note that by \eqref{t}, \eqref{u}, \eqref{v}, $dd^a$ induces the following action on $K^{(2)}$: $[t,u,v]\mapsto[t+u+v,u,v]$. Hence $(1+dd^a)(K^{(2)})=\{0,u+v\}$. Thus $\dim_{\mathbb{F}_2}R_2^{(2)}=2+2+1=5$.

Now, it follows from \eqref{x0}, \eqref{xm} that $(1+dd^a)(x_0)=(u,u)$, $(1+dd^a)(x_1)=u_2$ and $(1+dd^a)(x_m)=0$ for $m>1$, since from \eqref{more} we have for $m\geq 3$ that $u_m=[x_{m-3},x_{m-1}]$.

Also, from \eqref{um}, \eqref{vm}, $(1+dd^a)(u_m)=0$, $(1+dd^a)(v_m)=0$ for all $m\geq 0$, and $(1+dd^a)((u,u))=0$.

Thus $(1+dd^a)(R_2^{(2)})=\{0,(u,u)\}$, and so $\dim_{\mathbb{F}_2}R_3^{(2)}=2+5-1+1=7$.
Analogously, $(1+dd^a)(R_n^{(2)})=\langle (u,u), u_2\rangle$ for $n\geq 3$, and so by induction $\dim_{\mathbb{F}_2}R_{n+1}^{(2)}=2+\dim_{\mathbb{F}_2}R_n^{(2)}-2+1=n+5$.
\end{proof}

\begin{corollary}
$d(R_1)=3$, $d(R_2)=5$, $d(R_n)=n+4$ for $n\geq 3$.
\end{corollary}

\subsection{Groups $Q_n$}
Let $Q_1=B=\langle b, t, u, v\rangle=K\rtimes\{1,b\}$ and  $Q_n=(K\times R_{n-1})\langle b, (ac)^4\rangle$ for $n>1$.

\begin{proposition}\label{qn}
$Q_n=R_n\rtimes\{1,b\}$.
\[
Q_2=\langle b, x_0, u_0, v_0, (u,u)\rangle.
\]
\[
Q_n=\langle b, x_0, x_1, \dots, x_{n-2}, u_0, u_2, v_{n-2}, (u,u)\rangle.
\]
for $n\geq 3$.
\end{proposition}
 \begin{proof}
 The proof is analogous to the proof of Proposition \ref{rn}, using lists of generators for $R_n$ from the statement of Proposition \ref{rn}, and the additional fact that $x_0^b=u_1x_0$, hence $u_1=x_0^bx_0$.
 \end{proof}

 \begin{proposition}
 $\dim_{\mathbb{F}_2}Q_n^{(2)}=n+4$ for $n\geq 3$, $\dim_{\mathbb{F}_2}Q_2^{(2)}=5$, and $\dim_{\mathbb{F}_2}Q_1^{(2)}=4$.
 \end{proposition}
 \begin{proof}
Note that to compute $\dim_{\mathbb{F}_2}Q_n^{(2)}$ we may use the Lemma \ref{lem}, since $Q_n=R_n\rtimes\{1,b\}$. Thus we need to compute the induced action of $b$ on $R_n^{(2)}$.

 Note that $(1+b)(K^{(2)})=0$, by \eqref{t}, \eqref{u}, \eqref{v}. Also, $(1+b)(x_0)=u_1$, $(1+b)(x_m)=0$ for $m\geq 1$ by \eqref{xxm}. $(1+b)(u_m)=0$ for $m\geq 0$ by \eqref{uum}. Finally, $(1+b)((u,u))=(u,u)+(v,1)+(1,u)$, and since $(t,1)^{x_0}=(v,1)^{-1}(t,1)(u,1)$, it follows that $(v,1)+(u,1)=0$, and thus $(1+b)((u,u))=(u,u)+(v,1)+(1,u)=0$.

 Hence $(1+b)(R_n^{(2)})=\langle u_1\rangle$, and thus  $\dim_{\mathbb{F}_2}Q_2^{(2)}=5-1+1=5$, and $\dim_{\mathbb{F}_2}Q_n^{(2)}=n+4-1+1=n+4$.
 \end{proof}

 \begin{corollary}
$d(Q_1)=4$, $d(Q_2)=5$, $d(Q_n)=n+4$ for $n\geq 3$.
\end{corollary}
Finally we  introduce   groups $P_n$.

 \subsection{Groups $P_n$}
 Let $P_n=st_{\mathcal G}(1^n)$. Then $P_1=st_{\mathcal G}(1)$, and $P_n=(K\times Q_{n-1})\langle c, (ac)^4\rangle$ for $n>1$.

 \begin{proposition}\label{gener0}
 $P_n=R_n\rtimes\{1,b,c,d\}$ and
 \[P_n=\langle c, d, x_0, x_1,\dots, x_{n-2}, u_0, v_{n-2}, (u,u)\rangle,
 \] for $n\geq 2$.
 \end{proposition}
 \begin{proof}
 Analogous to the proofs of Proposition \ref{rn} and Proposition \ref{qn}, using the additional fact that $x_1^d=u_2x_1$ by \eqref{xxxm}.
 \end{proof}

 \begin{proposition}
 \begin{equation}
 \label{gener}
 \dim_{\mathbb{F}_2}P_n^{(2)}=n+4
 \end{equation}
  for $n\geq 2$, and $\dim_{\mathbb{F}_2}P_1^{(2)}=4$.

 \end{proposition}
 \begin{proof}
 For $n=1$ it follows from $P_1=st(1)=\langle d,c,d^a,c^a\rangle$.
 For $n\geq 2$ it follows from $P_n=R_n\rtimes\{1,b,c,d\}$ and a slight generalization of the Lemma \ref{lem}: $P_n^{(2)}=(R_n^{(2)}/V_n)\oplus\mathbb{F}_2^2$, where $V_n=(1+b)R_n^{(n)}+(1+c)R_n^{(n)}$.
 \end{proof}

Now  we  are  ready  to  prove  theorem  \ref{th2}.  As  $\mathcal G$ acts  level transitive  on $\mathcal T$, for  each  vertex $v$ of  the  level  $n$ the   group
$st_{\mathcal G}(v)$  is  conjugate  in   $\mathcal G$  to  $P_n$.   From  Proposition \ref{gener0} it  follows  that  the minimum  number  of  generators  of $P_n$ is $\leq  n+4$.  The  last  proposition  shows  that  it  is  exactly  $n+4$  when $n \geq 2$.  As  index of  $P_n$  in $\mathcal G$   is  $2^n$   (because  of the level  transitivity  of  $\mathcal G$)  the  ratio

\[
\frac{d(H)-1}{[\mathcal{G}:H]}=\frac{n+3}{2^n}
\]

for  $H=st_{\mathcal G}(v)$ where  vertex $v$ belongs  to  the  level  $n\geq 2$. If $v$  is  a  vertex  of  the  first  level  then  $H=st_{\mathcal G}(1)$  is a  subgroup  of  index  2  and is $4$-generated  (by  elements  $b,c,b^a,c^a$).  Taking  all  this  into  account we  get  the  conclusion  of the  theorem  \ref{th2}.

\bibliographystyle{plain}
\bibliography{mylib}

\end{document}